\documentclass{amsart}

\usepackage{amsmath,amscd,amssymb}

\usepackage{mathtools}
\usepackage{subcaption}
\usepackage[pagebackref,colorlinks,citecolor=blue,linkcolor=magenta]{hyperref}

\usepackage[capitalize]{cleveref}
\usepackage{tikz}
\usepackage{centernot}
\usepackage{enumitem}
\usepackage[utf8]{inputenc}

\newcommand{\GG}{\mathcal{G}}
\newcommand{\HH}{\mathcal{H}}
\newcommand{\DD}{\mathcal{D}}

\newcommand{\spans}{\operatorname{span}}
\renewcommand{\ne}{\operatorname{ne}}
\newcommand{\child}{\mathfrak{c}}
\newcommand{\parent}{\mathfrak{p}}

\DeclareMathOperator{\pa}{pa}
\DeclareMathOperator{\ch}{ch}
\DeclareMathOperator{\de}{de}

\DeclareMathOperator{\nd}{nd}
\DeclareMathOperator{\cl}{cl}
\DeclareMathOperator{\BIC}{BIC}
\DeclareMathOperator{\diam}{diam}

\newcommand\independent{\protect\mathpalette{\protect\independenT}{\perp}}
\def\independenT#1#2{\mathrel{\rlap{$#1#2$}\mkern2mu{#1#2}}}

\def\newop#1{\expandafter\def\csname #1\endcsname{\mathop{\rm
#1}\nolimits}}
\newop{conv}
\newop{CIM}

\newtheorem{theorem}{Theorem}[section]
\newtheorem{proposition}[theorem]{Proposition}
\newtheorem{corollary}[theorem]{Corollary}
\newtheorem{lemma}[theorem]{Lemma}

\theoremstyle{definition}
\newtheorem{definition}[theorem]{Definition}
\newtheorem{example}[theorem]{Example} 

\theoremstyle{remark}

\newtheorem{conjecture}{Conjecture}

 \ifpdf
 \hypersetup{
   pdftitle={On the Diameter of the Characteristic Imset Polytopes},
   pdfauthor={P. Restadh}
 }
 \fi

\title{Diameters of the Characteristic
Imset Polytopes}

\author{Petter Restadh}

\email{petterre@kth.se}

\address{Department of Mathematics\\
    KTH Royal Institute of Technology\\
    SE-100 44 Stockholm, Sweden}


\begin{document}

%%%%%%%%%%%%%%%%%%%%%%%%%%
%---ABSTRACT
\begin{abstract}
It has been shown that the edge structure of the characteristic imset polytope is closely connected to the question of causal discovery. 
The diameter of a polytope is an indicator of how connected the polytope is and moreover gives us a hypothetical worst case scenario for an edge-walk over the polytope. 
We present low-degree polynomial bounds on the diameter of $\CIM_n$ and, for any given undirected graph $G$, the face $\CIM_G$. 
\end{abstract}
\maketitle

%%%%%%%%%%%%%%%%%%%%%%%%%%
%---SECTION: Intro
\section{Introduction}
\label{sec: intro}
Several algorithms within causal discovery were recently discovered to be edge-walks along convex polytopes. 
A natural question becomes how efficient such an edge-walk can become. 
To this end we study the diameters of these polytopes.

Let $[n]\coloneqq \{1, \dots,n\}$ and $\GG=([n], E)$ be a directed acyclic graph (DAG). 
The characteristic imset of $\GG$, $c_\GG$, is a 0/1-vector, indexed by subsets of $[n]$, that in coordinate  $S\in \{S\subseteq [n], |S|\geq 2\}$ assumes the value  
\[
c_\GG(S) \coloneqq
\begin{cases}
1	&	\text{ if there exists $i\in S$ such that, $S\subseteq \pa_\GG(i)\cup \{i\}$},	\\
0	&	\text{ otherwise}.	\\
\end{cases}
\]
Then we define the \emph{characteristic imset polytope} as
\[
\CIM_n\coloneqq \conv\left(c_\GG\colon \GG=([n], E) \textrm{ a DAG}\right).
\]
The polytope $\CIM_n$ is a full dimensional ($\dim\CIM_n=\left|\{S\subseteq [n], |S|\geq 2\}\right|=2^n-n-1$) polytope whose vertices are precisely the characteristic imsets of DAGs. 
The mapping $\GG\mapsto c_\GG$ is not injective; we do however have a clear graphical understanding of when $c_\GG=c_\HH$ (see \cref{lem: studeny}).
It has been shown that $\CIM_n$ has many facets, at least one for each connected matroid on $[n]$ \cite{S15}, but a complete facet description is only available for $n\leq 4$. 
Especially we are interested in explaining the geometry, such as the edges or facets, of $\CIM_n$ in terms of the DAGs. 

The motivation for these questions comes from the area of causal discovery where a well-studied question regards finding  algorithms for inferring a DAG from data \cite{C02, SG91, WSYU17}. 
To do this we interpret $i\to j$ in $\GG$ to mean $i$ being a direct cause of $j$. 
Studen\'y, Hemmecke, and Lindner transformed this question into a linear program over $\CIM_n$ \cite{S05, SHL10}. 
The authors of \cite{LRS20} showed that the edge structure of $\CIM_n$ is of particular interest. % \cite{LRS20}. 
In their paper we are given a geometric interpretation of several greedy algorithms as edge-walks over $\CIM_n$ and its faces. 
In particular, for any undirected graph $G$, the face \cite{LRS20}
\[
\CIM_G\coloneqq \conv\left(c_\GG\colon \GG=([n], E) \textrm{ a DAG with skeleton }G\right),
\]
(of $\CIM_n$) is studied.
A complete characterisation of the edges of $\CIM_G$ when $G$ is a tree or a cycle was recently discovered \cite{LRS22}. 
For general $G$, less is known and the only edges with a clear interpretation are the ones given in \cite[Proposition 3.2]{LRS20}, namely that changing the direction of a single edge in a DAG, that does not create a directed cycle, gives us an edge in $\CIM_G$ (see \cref{thm: turn pair}). 

Given any polytope $P$, the vertex-edge graph of $P$, $G(P)$, is the graph with nodes corresponding to the vertices of $P$ and an edge $v-u\in G(P)$ if and only if $\conv(v,u)$ is an edge of $P$. 
We define the distance in $G(P)$ between $v$ and $u$ as the length of the shortest path between $v$ and $u$ in $G(P)$, and the diameter of $P$, $\diam(P)$, as the maximal distance between any two vertices in $G(P)$. 
Polytope diameters have been studied extensively and show up in several different contexts \cite{Naddef89, Santos12, Zie95}. 
The original motivation was that the diameter of a polytope is a lower bound on the number of steps a simplex-type algorithm must take. 
It also provides an indication of whether the graph $G(P)$ is sparse or more densely connected. 
Therefore, to better understand the polyhedral aspects of causal discovery knowledge of the diameter of $\CIM_n$ and it's faces is desired. 

In this paper we will establish low-degree polynomial bounds on the diameters of the above-mentioned polytopes. 
In particular, in \cref{sec: diameter} we see that for a general undirected graph $G=([n], E)$ we have $\diam\CIM_G\leq |E|$.
If $G$ is a tree we can, via the work of \cite{LRS22}, improve this bound to $\diam\CIM_G\leq n-2$ and give a lower bound in terms of the maximal path length of $G$. 
This is done in \cref{subsec: trees}. 
Finally in \cref{subsec: cim p} we show, using a new type of edge, that $\diam\CIM_n\leq 2n-2$. 
Given this linear upper bound, we conjecture that we in fact have linear upper bounds (in $n$) on $\diam \CIM_G$ for any $G$. 

%%%%%%%%%%%%%%%%%%%%%%%%%%%%%%%
%---SUBSECTION: Background
\subsection{Background}
Let $\GG$ be a DAG. 
If we have $i\to j$ in $\GG$ we say that $i$ is a \emph{parent} of $j$, or that $j$ is a \emph{child} if $i$. 
The set of all parents and children of a node $i$ is denoted with $\pa_\GG(i)$ and $\ch_\GG(i)$, respectively. 
For any joint distribution $P$ over $X_1, \dots, X_n$ we say that $P$ is \emph{Markov} to $\GG$ if $P$ entails the conditional independence statements
\[
X_i\independent X_{\nd_\GG(i)\setminus \pa_\GG(i)}|X_{\pa_\GG(i)}
\]
for all $i\in[n]$ where $X_S$ denotes the set $\{X_i\}_{i\in S}$ for any $S\subseteq [n]$.
Here $\nd_\GG(i)$ denotes the non-descendants of $i$ in $\GG$, that is all vertices $j$ such that there does not exist a directed path $j\to\dots\to i$. 
Informally this should be interpreted as "the only direct causes of $i$ are the parents of $i$".
It can happen that two DAGs encode equivalent conditional independence statements (see \cref{ex: markov equivalence}). 
If this is the case we call them \emph{Markov equivalent} or that they belong in the same \emph{Markov equivalence class (MEC)}.

\begin{example}
\label{ex: markov equivalence}
In \cref{fig: ex markov equivalence} we have 3 examples of DAGs and the conditional independence statements encoded by them.
Note that two of them encode exactly the same conditional independence statements.
\begin{figure}
\[
\begin{tikzpicture}[scale =0.6]
\node (m1) at (0,2) {$X_1$};
\node (d1) at (2,2) {$X_3$};
\node (c1) at (1,0) {$X_2$};
\draw[->] (m1) -- (c1);
\draw[->] (d1) -- (c1);
\node at (1, -2) {$\{ X_1 \independent X_3| \emptyset\}$};

\node (g2) at (6,3) {$X_3$};
\node (p2) at (6,1.5) {$X_2$};
\node (c2) at (6,0) {$X_1$};
\draw[->] (g2) -- (p2);
\draw[->] (p2) -- (c2);
\node at (6, -2) {$\{ X_1 \independent X_3| X_2\}$};

\node (p3) at (11,2) {$X_2$};
\node (b3) at (10, 0) {$X_1$};
\node (s3) at (12,0) {$X_3$};
\draw[->] (p3) -- (b3);
\draw[->] (p3) -- (s3);
\node at (11, -2) {$\{ X_1 \independent X_3|X_2\}$};

\end{tikzpicture}
\]
\caption{An example of 3 DAGs and the conditional independence statements encoded by them.}
\label{fig: ex markov equivalence}
\end{figure}
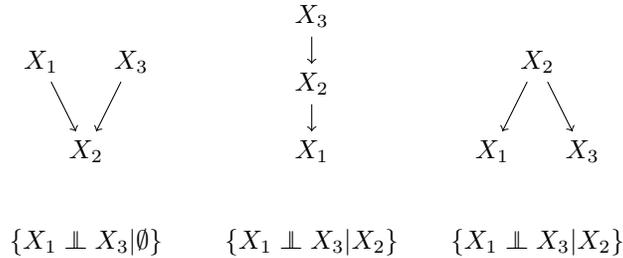
\end{example}

An induced subgraph of $\GG$ such that $\GG|_{\{i,j,k\}}=i\to j\leftarrow k$ is called a \emph{v-structure}. 
The undirected graph $G$ that shares the same vertices and adjacencies as $\GG$ is known as the \emph{skeleton} of $\GG$. 
We also let $\ne_\GG(i)=\pa_\GG(i)\cup \ch_\GG(i)$, or equivalently $i$ is a neighbour of $j$ in a directed graph if they are neighbours in the skeleton,.
The \emph{closure} of $i$, denoted $\cl_\GG(i)\coloneqq \ne_\GG(i)\cup\{i\}$ is the set consisting of all neighbours of $i$ together with $i$. 
The following classical result by Verma and Pearl gives us a graphical interpretation of Markov equivalence.
%---THEOREM: Verma Pearl
\begin{theorem}
\label{thm: verma pearl}
\cite{VP92}
Two DAGs are Markov equivalent if and only if they have the same skeleton and the same v-structures.
\end{theorem}

The previously mentioned work by Studen\'y \cite{S05} outlines how to encode a model characterized via conditional independence statements, for example DAG models, via integer vectors. 
One key idea being that a maximum likelihood estimation over all models is equivalent to maximizing a linear function over the set of vectors. 
Developing this idea Studen\'y, Hemmecke, and Lindner introduced the characteristic imset; a vector encoding where the MEC of $\GG$ could easily be recovered from $c_\GG$. 
Indeed, it is direct from the definition of the characteristic imset that the following holds:
%---LEMMA: Imset Structure
\begin{lemma}
\label{lem: imset structure}
\cite{SHL10}
Let $\GG$ be a DAG with nodes $[n]$.
Then for any distinct nodes $i$, $j$, and $k$ we have
\begin{enumerate}[label=(\arabic*)]
\item{$i\leftarrow j$ or $i\rightarrow j$ in $\GG$ if and only if $c_\GG(\{i, j\})=1$.}
\item{$i\rightarrow j \leftarrow k$ is a v-structure in $\GG$ if and only of $c_\GG(\{i,j,k\})=1$ and $c_\GG(\{i,k\})=0$.}
\end{enumerate}
\end{lemma}
That is, the characteristic imset encodes the skeleton and the v-structures of a DAG. 
Therefore it is easy to recover the MEC of $\GG$ from $c_\GG$. 
Moreover, Studen\'y, Hemmecke, and Lindner showed that the characteristic imset is in fact a unique representation of the MEC.
%----LEMMA: Studeny
\begin{theorem}
\label{lem: studeny}
\cite{SHL10}
Two DAGs $\GG$ and $\HH$ are Markov equivalent if and only if $c_\GG = c_\HH$.
\end{theorem}
It is then direct from \cref{lem: imset structure} that the characteristic imset is encoded in the sets of size 2 and 3. 
%---LEMMA: 2/3 sets characterize the imset
\begin{corollary}
\label{lem: lindner}
\cite[Corollary 2.2.6]{L12} %Corollary 2.2.6
Two characteristic imsets $c_\GG$ and $c_\HH$ are equal if and only if $c_\GG(S)=c_\HH(S)$ for all sets $S$ such that $|S|\in\{2,3\}$.
\end{corollary}

An edge $i\to j\in\GG$ is \emph{essential} if we have $i\to j\in\DD$ for all $\DD$ in the MEC of $\GG$. 
In this case $i$ is an \emph{essential parent} of $j$, and $j$ is an \emph{essential child} of $i$.
Due to work of Andersson, Madigan, and Pearl \cite{AMP97}, the graphical properties of MECs are well-understood.
For a more thorough background on the statistical side of graphical models we refer to \cite{lauritzen1996, pearl2009causality, S01}. 

A central question within causal discovery is inferring a MEC from data. 
Than is, given i.i.d samples $\mathbf{D}$ from the joint distribution $P$ over $X_1, \dots, X_n$, find the MEC that best encodes the observed conditional independence statements in $\mathbf{D}$. 
This has often been interpreted as finding the DAG (or MEC of DAGs) maximising $\BIC(\GG, \mathbf{D})$, where $\BIC$ denotes the \emph{Bayesian information criterion} \cite{C02, LRS22, LRS20, TBA06}. 
Importantly the $\BIC$ is \emph{score equivalent} and \emph{decomposable} \cite{C02}, that is, it is a linear function over $\CIM_n$. 
Therefore, recovering the $\BIC$-optimal MEC from data can be phrased as a linear program \cite{S05, SHL10}.

Some of the best performing algorithms (GES \cite{C02}, MMHC \cite{TBA06}, and Greedy CIM \cite{LRS20}) were recently shown to be restricted edge-walks over $\CIM_n$ and its faces \cite{LRS20} (including $\CIM_G$). 
Computational data on $\CIM_4$ does however suggest that these edge-walks utilise very few out of all edges possible. 
This raises further question on how connected $G(\CIM_n)$ is and with that how feasible these methods are. 

%---SUBSECTION: Diameter----
\section{Diameter of \texorpdfstring{$\CIM$}{CIM} polytopes}
\label{sec: diameter}
The diameter of the polytope gives us an upper bound on the number of steps any edge-walk needs to perform to get from one vertex to another, assuming we are walking optimally. 
In \cite{LRS20}, it was shown that reversing an edge of a DAG $\GG$ either gives a Markov equivalent graph or an edge of $\CIM_G$. 
For any DAG $\GG$ with $i\to j\in \GG$ we define $\GG_{i\leftarrow j}$ to be the directed graph identical to $\GG$ but with the edge $i\to j$ reversed.  
\begin{theorem}\cite{LRS20}
\label{thm: turn pair}
Let $\GG$ be a DAG with skeleton $G$ and $i\to j\in\GG$. 
If $\GG_{i\leftarrow j}$ is a DAG not Markov equivalent to $\GG$, then $\conv(c_\GG, c_{\GG_{i\leftarrow j}})$ is an edge of $\CIM_G$.
\end{theorem}
Using this we can show an upper bound on the diameter of $\CIM_G$.

\begin{proposition}
Let $G=([n], E)$ be an undirected graph. Then $\diam(\CIM_G)\leq |E|$. 
\end{proposition}
Here we will use the same argument used in \cite[Proposition 3.5]{LRS20}.
\begin{proof}
Let $\GG$ and $\HH$ be two DAGs with skeleton $G$. 
We claim that we can always reverse at least one edge $i\to j$ in $\GG$, such that $i\leftarrow j$ is in $\HH$, and reach another DAG. 
As $\GG$ and $\HH$ share skeleton, and can thus at most differ on $|E|$ edges there exists a sequence of at most $|E|$ edges such that after reversing each we have a new DAG, and in the end we reach $\HH$.
From \cref{thm: turn pair} we have that each of these moves either correspond to an edge or a vertex of $\CIM_G$, and hence the result follows. 

To this end we let $\mathcal{C}$ be the set of edges that differ between $\GG$ and $\HH$. 
We impose a partial order on $\mathcal{C}$ as $i'\to j'\preceq i\to j$ if and only if $j'\in\de_\GG(j)$ or, if $j'=j$, $i\in\de(i')$. 
That is, we sort the children according to $\GG$ and the parents in reverse.
Let $i\to j$ be a maximal edge in $\mathcal{C}$. 
We claim that reversing $i\to j$ does not create a cycle. 

For the sake of contradiction, assume it does.
That is, we have another path $i\to \dots\to j$ in $\GG$. 
As every edge in this path is greater than $i\to j$, according to our partial order, all of these edges must be present in $\HH$. 
However $i\leftarrow j$ is in $\HH$, contradicting that $\HH$ was a DAG. 
Then the result follows.
\end{proof}
Computational data on random graphs ($n\leq 9$) does however suggest that the above proof utilise very few of the edges of $\CIM_G$.
In \cite{LRS22} all edges of $\CIM_G$ when $G$ is a tree were determined. 
% in this case this argument utilises very few of the total number of edges. 
Hence the upper bound $\diam \CIM_G\leq |E|$ can be improved, at least when $G$ in this case. 

%---SUBSECTION: Trees ----
\subsection{Trees}
\label{subsec: trees}
As all edges of $\CIM_G$ are known when $G$ is a tree we will begin with this case. 
First however, let us recall the results of \cite{LRS22}. 
Let $\GG$ and $\HH$ be two essential graphs with skeleton $G$ and assume that $G$ is a tree. 
Denote $N_i=\{S\cup \{i\}\colon S\subseteq \ne_G(i)\text{ and } |S|\geq 2\}$ and define 
\[
\Delta(\GG, \HH)\coloneqq \{i\in[n]\colon c_\GG|_{N_i}\neq c_\HH|_{N_i}\}.
\]
Equivalently $\Delta(\GG, \HH)$ is the set of all vertices $i$ such that $\GG$ has a v-structure at $i$ that $\HH$ does not have, or vice versa \cite{LRS22}. 
For any tree $G=([n], E)$ and subset $S\subseteq [n]$ we denote with $\spans(S)$ the vertices of the unique spanning tree of $S$ in $G$. 
Moreover we say that $i\in G$ is an \emph{internal node} of $G$ if $i$ is not a leaf, and the graph induced by $G$ on the internal nodes is denoted $G^\circ$.  

%---DEFINITION: Essential flip----
\begin{definition}
[Essential flip]
\label{def: essential flip}
Let $\GG$ and $\HH$ be two non-Markov equivalent essential graphs with skeleton $G$, a tree, and denote $\Delta=\Delta(\GG, \HH)$. 
Assume that both $\GG|_{\spans(\Delta)}$ and $\HH|_{\spans(\Delta)}$ do not contain any undirected edges. 
Assume moreover that each edge of $\GG$ and $\HH$ differ on $G|_{\spans(\Delta)}$. % and that $G|_{\spans(\Delta)}$ is connected.
Then we say that the pair $\{\GG, \HH\}$ is an essential flip. 
\end{definition}
The importance of essential flips is shown in the following theorem. 
\begin{theorem}\cite{LRS22}
\label{thm: edges of trees characterization}
If $G$ is a tree, then $\conv(c_\GG, c_\HH)$ is an edge of $\CIM_G$ if and only if the pair $\{\GG, \HH\}$ is an essential flip. 
\end{theorem}
As we prefer to work with DAGs as opposed to essential graphs we also have an alternative characterization. 

\begin{theorem}
\label{thm: subtree condition}
Suppose that $\GG$ and $\HH$ are DAGs on the same skeleton $G$ that is a tree. Assume the edges that differ between $\GG$ and $\HH$ form a subtree $T$ of $G$. 
Suppose further that $\Delta(\GG,\HH)\neq\emptyset$. Then the essential graphs of $\HH$ and $\GG$ form an essential flip if and only if each internal node $i$ of $T$ satisfy the conditions given below. 
We use notation $\{\child_i\}=T\cap \ch_\GG(i)$ and $\{\parent_i\}=T\cap \pa_\GG(i)$, when those are unique.

\smallskip
\noindent
\begin{tabular}{| l | l | l | l |}
\hline
& $| T\cap \pa_\GG(i)| $ & $|T\cap \ch_\GG(i)| $ & Local criteria for $\GG$ and $\HH$ to form essential flip\\
\hline
\hline
I & $\ge 2$ &$\ge 2$ & \\
\hline
II & $\ge 2$ &$0$ & \\
\hline
III& $0$ & $\ge 2$& \\
\hline
IV & $\ge 2$ &$1$ & $ | \pa_\GG(i)\setminus T|\ge 1$ or \\
 & & & if $\exists$ v-structure at $\child_i$ in $\GG$, then \\
 &&& $\child_i$ has essential parent in $\HH$ \\
\hline
V & $1$ &$\ge 2$ & $ | \pa_\GG(i)\setminus T|\ge 1$ or \\
 & & & if $\exists$ v-structure at $\parent_i$ in $\HH$, then \\
 &&& $\parent_i$ has essential parent in $\GG$ \\
\hline
VI & $1$ &$1$ & if there are nodes of  $\Delta$ in both\\ 
&&& connected components of $T\setminus \{i\}$\\
& & & then $|\pa_\GG(i)\setminus T|\ge 1$ or  $\child_i$\\
%& & & $\exists$ v-structure at $j$ in $\GG$ or v-structure at $k$ in $\HH$ but if \\
& & & has essential parent in $\HH$ and $\parent_i$ \\ 
&&& has essential parent in $\GG$.\\
\hline
\end{tabular}
\end{theorem}

As we know all edges of the $\CIM_G$ polytope we expect to find a good bound on the diameter.

%---CLAIM: Diameter Trees----
\begin{proposition}
\label{prop: max tree diam}
If $G$ is a tree, then the diameter of $\CIM_G$ is less than or equal the number of internal vertices of $G$. 
\end{proposition}

To show this we will need a quick lemma that follows from \cref{thm: subtree condition}.

\begin{lemma}
\label{lem: valid tree extra node}
For any given DAG $\GG$ with skeleton $G$ and $i\in G^\circ$ an internal node, assume $T$ is a subtree of $G$.
Let $\GG'$ be the DAG on one extra node, $n+1$, and the edge $n+1\leftarrow i$.
Define $\HH$ and $\HH'$ to be identical to $\GG$ and $\GG'$ except we change the direction of every edge in $T$. 
Then if $\{\GG, \HH\}$ is an essential flip, so is $\{\GG', \HH'\}$. 
\end{lemma}

\begin{proof}
By definition $T$ does not contain the edge $(n+1) - i$. 
If $n+1\leftarrow i$ then the edge $(n+1) - i$ is not part of any v-structure, hence any other edge in $\GG$ is essential if and only if it is essential in $\GG'$.
Thus the result follows directly from \Cref{thm: subtree condition}. 
\end{proof}

\begin{proof}[Proof of \Cref{prop: max tree diam}]
Inductively assume we can transform any directed tree with fewer internal nodes to any other directed three with the same skeleton using only the transformations of \Cref{thm: subtree condition} in at most the number of steps equal to the number of internal nodes. 

Let $\GG$ and $\HH$ be two graphs with skeleton $G$ and let $m$ denote the number of internal vertices. 
It is enough to find DAGs $\GG=\GG_0,\GG_1, \GG_2,\dots, \GG_m=\HH$ such that each pair $\{\GG_k, \GG_{k+1}\}$ is an essential pair or are Markov equivalent, for all $0\leq k\leq m-1$.  
If $\GG$ and $\HH$ are Markov equivalent we are done. 
Let $r_1\in \Delta(\GG, \HH)$.
Notice that $\Delta(\GG, \HH)\subseteq G^\circ$. 
We can imagine $G^\circ$ to be rooted at $r_1$, and thus we will transform vertices from the root and onward. 

Define $\GG_1'$ to be the DAG identical to  $\GG$ outside of $\cl_G(r_1)$ and direct $\cl_G(r_1)$ as in $\HH$. 
Then we have three cases, either $\GG_1'$ is Markov equivalent to $\GG$, $\{\GG_1', \GG\}$ is an essential flip, or $\GG_1'$ and $\GG$ differ on a subtree that is within cases IV, V, or VI of \Cref{thm: subtree condition}. 
In the first and second case we can choose $\GG_1=\GG_1'$. 
In the third case we have three subcases, either we are in case IV, V, or VI. 

If we are in case IV and $\{\GG_1', \GG\}$ is not an essential flip there must be a unique child of $r_1$, say $c$, and there is a v-structure at $c$. 
Then we define $\GG_1$ to be the DAG identical to $\GG_1'$ except flip the edge $r_1\leftarrow c$, that is we preserve the direction of the edge as in $\GG$. 
Then by \Cref{thm: subtree condition} $\{\GG_1, \GG\}$ will be an essential flip. 
Notice that $\pa_{\GG_1}(r_1)=\emptyset$ in this case.

 If we are in case V and $\{\GG_1', \GG\}$ is not an essential flip there must be a unique parent of $r_1$, say $a$ with a unique non essential parent $a'$. 
 Since the edge is non-essential we can direct the subtree in $G\setminus \{r_1\}$ containing $a$, call it $G_{a}$, such that we have $a\to a'$ and we do not change any v-structures. 
 Then we can define $\GG_1$ to be the DAG where we direct the subtree $G_{a}$ as described above, and the rest as $\GG_1'$. 
 Then $\{\GG_1, \GG\}$ will be an essential flip.
 If we are in case VI then we either have $|\pa_\GG(r_1)\setminus T|\geq 1$, in which case $\{\GG_1', \GG\}$ is an essential flip, or $r_1\notin \Delta(\GG, \HH)$, a contradiction from how we choose $r_1$. 
 Notice that regardless of case we obtain $c_{\GG_1}|_{N_{r_1}}=c_{\HH}|_{N_{r_1}}$ and we never change the direction of any edge in $\cl_G(r_1)$ that is already directed as in $\HH$. 

 Let $r_2$ be any node adjacent to $r_1$ in $G^\circ$. 
 If we have $r_2\to r_1$ we either have $r_2\to r_1$ in $\HH$, in which case the result follows from the induction hypothesis and \Cref{lem: valid tree extra node} via considering the subtree containing $r_2$ in $G\setminus\{r_1\}$, or we changed to $\GG_1$ from case three, subcase IV. 
 However, $r_1$ then has $r_2$ as a single parent, and hence the result follows via using the induction hypotheses on the same subtree as before, with the node $r_1$ added. 

 If we have $r_2\leftarrow r_1$ we can repeat the construction as before, except we will use that $r_1\in \pa_{\GG_1}\setminus T$, instead of $r_2\in\Delta(\GG_1, \GG_2)$. 
 Indeed, by the above construction we only have $r_2\leftarrow r_1$ in $\GG_1$ if $r_2\leftarrow r_1$ in $\HH$. 
 Then case one and two follows as before and case three, subcases IV and V are identical.
 However, case three, subcase VI is now true as $\pa_{\GG_1}\setminus T$. 

 Hence we can always continue our construction one edge per internal node. 
 It follows that we at most need $m$ steps to transform $\GG$ to $\HH$. 
 \end{proof}

 \begin{example}
 In \Cref{fig: tree diameter construction} we have repeated the construction we did in the proof of \Cref{prop: max tree diam}. 
 Note however that it is possible to move between the same MECs in two steps, as seen in \Cref{fig: tree diameter opt}.
 \end{example}

 \begin{figure}
 \[
 \begin{tikzpicture}[scale=0.8]
 \begin{scope}
 \node (a) at (0,0) {$a$};
 \node (b1) at (0:1) {$b_1$};
 \node (b2) at (90:1) {$b_2$};
 \node (b3) at (180:1) {$b_3$};
 \node (b4) at (-90:1) {$b_4$};
 \node (c1) at (0:2) {$c_1$};
 \node (c2) at (90:2) {$c_2$};
 \node (c3) at (180:2) {$c_3$};
 \node (c4) at (-90:2) {$c_4$};

 \foreach \from/\to in {a/b1, a/b2, a/b3, a/b4, b1/c1, b2/c2, b3/c3, b4/c4}{
     \draw[->] (\from) -- (\to);
 }
 \end{scope}
 \begin{scope}[shift={(2.5, -5.5)}]
 \node (a) at (0,0) {$a$};
 \node (b1) at (0:1) {$b_1$};
 \node (b2) at (90:1) {$b_2$};
 \node (b3) at (180:1) {$b_3$};
 \node (b4) at (-90:1) {$b_4$};
 \node (c1) at (0:2) {$c_1$};
 \node (c2) at (90:2) {$c_2$};
 \node (c3) at (180:2) {$c_3$};
 \node (c4) at (-90:2) {$c_4$};

 \foreach \from/\to in {a/b1, a/b2, a/b3, a/b4, c1/b1, b2/c2, b3/c3, b4/c4}{
     \draw[->] (\from) -- (\to);
 }
 \end{scope}

 \begin{scope}[shift={(5.5,0)}]
 \node (a) at (0,0) {$a$};
 \node (b1) at (0:1) {$b_1$};
 \node (b2) at (90:1) {$b_2$};
 \node (b3) at (180:1) {$b_3$};
 \node (b4) at (-90:1) {$b_4$};
 \node (c1) at (0:2) {$c_1$};
 \node (c2) at (90:2) {$c_2$};
 \node (c3) at (180:2) {$c_3$};
 \node (c4) at (-90:2) {$c_4$};

 \foreach \from/\to in {a/b1, a/b2, a/b3, a/b4, c1/b1, c2/b2, b3/c3, b4/c4}{
     \draw[->] (\from) -- (\to);
 }
 \end{scope}
 
 \begin{scope}[shift={(8.5,-5.5)}]
 \node (a) at (0,0) {$a$};
 \node (b1) at (0:1) {$b_1$};
 \node (b2) at (90:1) {$b_2$};
 \node (b3) at (180:1) {$b_3$};
 \node (b4) at (-90:1) {$b_4$};
 \node (c1) at (0:2) {$c_1$};
 \node (c2) at (90:2) {$c_2$};
 \node (c3) at (180:2) {$c_3$};
 \node (c4) at (-90:2) {$c_4$};

 \foreach \from/\to in {a/b1, a/b2, a/b3, a/b4, c1/b1, c2/b2, c3/b3, b4/c4}{
     \draw[->] (\from) -- (\to);
 }
 \end{scope}
 \begin{scope}[shift={(11,0)}]
 \node (a) at (0,0) {$a$};
 \node (b1) at (0:1) {$b_1$};
 \node (b2) at (90:1) {$b_2$};
 \node (b3) at (180:1) {$b_3$};
 \node (b4) at (-90:1) {$b_4$};
 \node (c1) at (0:2) {$c_1$};
 \node (c2) at (90:2) {$c_2$};
 \node (c3) at (180:2) {$c_3$};
 \node (c4) at (-90:2) {$c_4$};

 \foreach \from/\to in {a/b1, a/b2, a/b3, a/b4, b1/c1, c2/b2, c3/b3, c4/b4}{
     \draw[->] (\from) -- (\to);
 }
 \end{scope}

\draw[thick, dashed] (1, -1) -- (1.5, -4.5)
    (3.5, -4.5) -- (4.5, -1) 
    (6.5, -1) -- (7.5, -4.5)
    (9.5, -4.5) -- (10, -1);
 
 \end{tikzpicture}
 \]
 \caption{An example of the construction in the proof of \Cref{prop: max tree diam}. The edges of $\CIM_G$ are marked with dashed lines. }
 \label{fig: tree diameter construction}
 \end{figure}
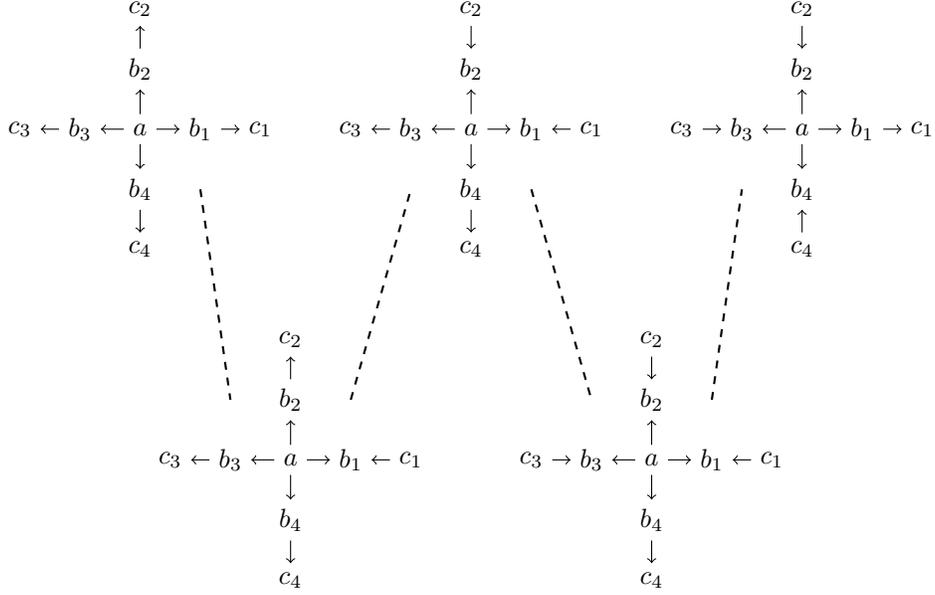

 \begin{figure}
 \[
 \begin{tikzpicture}[scale=0.8]
 \begin{scope}
 \node (a) at (0,0) {$a$};
 \node (b1) at (0:1) {$b_1$};
 \node (b2) at (90:1) {$b_2$};
 \node (b3) at (180:1) {$b_3$};
 \node (b4) at (-90:1) {$b_4$};
 \node (c1) at (0:2) {$c_1$};
 \node (c2) at (90:2) {$c_2$};
 \node (c3) at (180:2) {$c_3$};
 \node (c4) at (-90:2) {$c_4$};

 \foreach \from/\to in {a/b1, a/b2, a/b3, a/b4, b1/c1, b2/c2, b3/c3, b4/c4}{
     \draw[->] (\from) -- (\to);
 }
 \end{scope}
 \begin{scope}[shift={(6,0)}]
 \node (a) at (0,0) {$a$};
 \node (b1) at (0:1) {$b_1$};
 \node (b2) at (90:1) {$b_2$};
 \node (b3) at (180:1) {$b_3$};
 \node (b4) at (-90:1) {$b_4$};
 \node (c1) at (0:2) {$c_1$};
 \node (c2) at (90:2) {$c_2$};
 \node (c3) at (180:2) {$c_3$};
 \node (c4) at (-90:2) {$c_4$};

 \foreach \from/\to in {b1/a, b2/a, b3/a, b4/a, b1/c1, b2/c2, b3/c3, b4/c4}{
     \draw[->] (\from) -- (\to);
 }
 \end{scope}

 \begin{scope}[shift={(12,0)}]
 \node (a) at (0,0) {$a$};
 \node (b1) at (0:1) {$b_1$};
 \node (b2) at (90:1) {$b_2$};
 \node (b3) at (180:1) {$b_3$};
\node (b4) at (-90:1) {$b_4$};
\node (c1) at (0:2) {$c_1$};
\node (c2) at (90:2) {$c_2$};
\node (c3) at (180:2) {$c_3$};
\node (c4) at (-90:2) {$c_4$};

\foreach \from/\to in {a/b1, a/b2, a/b3, a/b4, c1/b1, c2/b2, b3/c3, b4/c4}{
    \draw[->] (\from) -- (\to);
}

\draw[thick, dashed] (-9.5, 0) -- (-8.5, 0)
    (-3.5, 0) -- (-2.5, 0);
\end{scope}
\end{tikzpicture}
\]
\caption{A shorter path than the one constructed in the proof of \Cref{prop: max tree diam}. The edges of $\CIM_G$ are marked with dashed lines.}
\label{fig: tree diameter opt}
\end{figure}
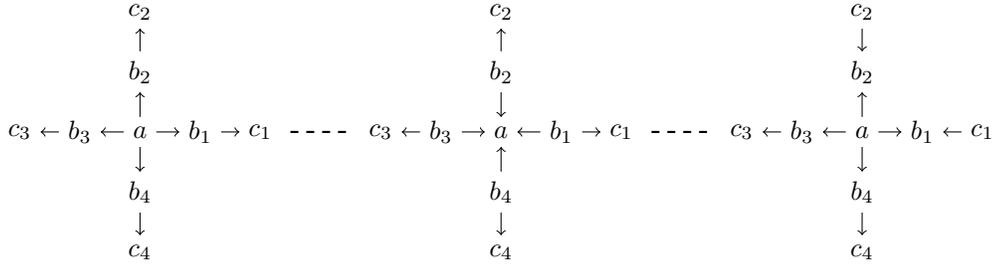
 
%---EXAMPLE: Diameter of path----
\begin{example}
\label{ex: diameter of path}
Let $G=I_n$ the path with $n$ vertices and let $\{\GG_i, \GG_{i+1}\}$ be an essential flip. 
Then the number of v-structures in $\GG_i$ and $\GG_{i+1}$ differ by at most one, as follows by the definition of essential flip. 
If we let $\GG$ be the graph without any v-structures and let $\HH$ be any graph with $\lfloor\frac{n-1}2\rfloor$ v-structures. 
It follows that the distance between $\GG$ and $\HH$ is at least, and in fact equal to, $\lfloor\frac{n-1}2\rfloor$ and thus the diameter of $\CIM_G$ is greater than or equal to $\lfloor\frac{n-1}2\rfloor$. 
To conclude equality it is enough to check that given any two adjacent internal nodes we can, in one move, make sure either one, or none, is a v-structure from any previous position. 
\end{example}

%---CLAIM: Minimum Diameter----
\begin{proposition}
\label{prop: min tree diam}
Let $G$ be a tree and let $p$ be the maximum path length of $G$. 
Then the diameter of $\CIM_G$ is at least $\left\lfloor \frac{p}2\right\rfloor$. 
\end{proposition}

Notice that the maximum path length is one less than the number of vertices in the graph. 
That is, for $I_n$ the maximum path length is $p=n-1$. 

\begin{proof}
% {\color{red} SLOPPILY WRITTEN, REWRITE?}
Let $P$ be a path of maximal length of $G$. 
Similar to \Cref{ex: diameter of path} it is enough to conclude that any essential flip changes the number of v-structures along $P$ by one.
Then it follows that the distance between two graphs, one with no v-structures in $P$ and one with $\left\lfloor \frac{p}2\right\rfloor$ v-structures, is at least $\left\lfloor \frac{p}2\right\rfloor$. 

Thus given any DAG $\GG$ with skeleton $G$ and a subtree $T'$ fulfilling the conditions of \Cref{thm: subtree condition}. 
Define $T=P\cap T'=v_0-v_1- \dots- v_{k}$ and if $v_0$ and/or $v_k$ are not the endpoints of $P$ we also consider the extra vertices $\alpha$ and $\beta$ defined as $\alpha - v_0-v_1- \dots- v_{k} -\beta\subseteq P$.
Notice that $\alpha$ and $\beta$ might not exist, but that can be thought of as we have $\alpha\leftarrow v_0$ or $v_k\to \beta$ in $G$. 
Moreover, we can assume all arrows in $T\cup\{\alpha, \beta\}$ are reversed.
A priori reversing all edges in $T\cup\{\alpha, \beta\}$ could lead to additional v-structures involving $\alpha$ and $\beta$, however by \cref{thm: subtree condition} we can safely ignore them. 
In this connected part we remove a v-structure whenever we have $i\to j\leftarrow k$ for $i,j,k\in T\cup \{\alpha, \beta\}$ and we add a v-structure whenever we have $i\leftarrow j\to k$, note that these patterns must be interlacing along  $\alpha - v_0-v_1- \dots- v_{k} -\beta$. 
Hence the number of v-structures can differ with at most one and the result follows. 
\end{proof}

Thus the combination of \cref{prop: max tree diam} and \cref{prop: min tree diam} implies the following theorem.

%%%%%%%%%%%%%%%%%%%%%%%%%%%
%---THEOREM: Diam tree------
\begin{theorem}
\label{thm: diam tree}
Let $G$ be a tree with $m$ internal nodes and maximum path length $p$. 
Then $\left\lfloor\frac{p}2\right\rfloor\leq\diam(\CIM_G)\leq m$. 
\end{theorem}
Notice that for all trees we have $m\leq n-2$ and for paths we have $p+1=n=m-2$, therefore the diameter of $\CIM_{I_n}$ is linear in $n$. 
Hence the diameter of $\CIM_G$ has a worst case scenario of growing linearly in $n$, when $G$ is a tree.
There are however classes of trees where our lower bound is constant in $n$; if $G$ is a star the above gives us $1\leq \diam\CIM_G\leq 1$, which is consistent with a result of \cite{LRS22} telling us that $\CIM_G$ is a simplex in this case and hence $\diam\CIM_G=1$. 
Computational results on random trees ($n\leq 9$) suggest that our lower bound is tight, while our upper bound is not.

From the perspective of random trees, we expect $m\approx(1-e^{-1})n$ internal nodes (with respect to the uniform distribution), while the expected maximum path length is $p\leq C\sqrt{n}$ for some constant $C$ \cite{ANS22, RS67}.
Further investigation on the expected diameter of $\CIM_G$, for a random tree $G$, would be very interesting.

%%%%%%%%%%%%%%%%%%%%%%%%%%%%%%%%%
%---SUBSECTION: The whole polytope------
\subsection{The Whole Polytope}
\label{subsec: cim p}
Up until now we have discussed the diameter of faces of $\CIM_n$, but \cite{LRS20} also gives us edges that are not in $\CIM_G$ for any $G$. 
If $\GG$ is a DAG with $i$ and $j$ not adjacent in the skeleton of $\GG$. 
Then we denote with $\GG_{+i\leftarrow j}$ the directed graph identical to $\GG$ with the edge $i\leftarrow j$ added. 
Notice that $\GG$ and $\GG_{+i\leftarrow j}$ have different skeleton and hence are never Markov equivalent. 
\begin{theorem}\cite{LRS20}
\label{thm: edge pair}
Let $\GG$ be a DAG and $i$ and $j$ be non adjacent vertices. If $\GG_{+i\leftarrow j}$ is a DAG, then $\conv(c_\GG, c_{\GG_{+i\leftarrow j}})$ is an edge of $\CIM_n$.  
\end{theorem}
Applying the above theorem directly gives us an upper bound of the diameter of $2\binom{n}{2}$. 
Indeed, from the empty graph we can walk to any DAG via adding in the correct edges one-by-on, and this requires at most $\binom{n}{2}$ number of steps. 
However, we can show a better bound utilising a new type of edge. 

\begin{proposition}
\label{prop: realizing a set}
Let $i\in [n]$, $S^\ast\subseteq [n]\setminus\{i\}$ and let $\GG$ be a DAG such that $\GG|_{S^\ast}$ is the empty graph. 
Let $\HH$ be the graph identical to $\GG$ but with all the edges $j\to i$ for $j\in S^\ast$. 
Then if $\HH$ is a DAG, $\conv(c_\GG,c_\HH)$ is an edge of $\CIM_n$. 
\end{proposition}

\begin{proof}
Notice that since $\GG\subseteq\HH$ we must have that $c_\HH(S)-c_\GG(S)\geq 0$ for every $S$.
Moreover, if $|S^\ast| = 1 $ then the result follows by \cref{thm: edge pair}, thus we can assume that $|S^\ast|\geq 2$. 
We let 
\[
w(S)=\begin{cases}
n^2 &\text{if } c_\GG(S)=1,\\
-n^2 &\text{if } c_\HH(S)=0,\\
-1 &\text{if } |S|=2, c_\GG(S)=0, c_\HH(S)=1,\\
|S^\ast| &\text{if } S=S^\ast\cup \{i\}, \text{ and}\\
0 & \text{otherwise.}
\end{cases}
\]
Let $\DD$ be a DAG maximising $w^Tc_\DD$. 
Let $D$, $G$, and $H$ denote the skeleton of $\DD$, $\GG$, and $\HH$ respectively.
We begin by noticing that $|S^\ast|\leq n-1$ and hence we must have $c_\DD(S)=1$ for all $S$ such that $c_\GG(S)=1$ and $c_\DD(S)=1$ for all $S$ such that $c_\HH(S)=0$. 
As $G\subseteq H$ this with \cref{lem: lindner} gives us $G\subseteq D\subseteq H$. 

It is straightforward from the definition of the characteristic imset that $c_\GG(S)=1$ implies $c_\HH(S)=1$, and the other way around $c_\HH(S)=0$ implies $c_\GG(S)=0$. 
Then we have two cases.

If $c_\DD(S^\ast\cup \{i\})=0$ we must have $c_\DD(\{i,j\})=0$ for all $j\in S^\ast$ as otherwise $w^Tc_\DD\leq w^Tc_\GG-1<w^Tc_\GG$. 
Then we notice that $c_\GG$ and $c_\HH$ only differ in 3-sets of the from $\{i,j,k\}$ where $j\in S^\ast$ and $k\in\pa_\GG(i)\cup S^\ast$.  
For every such 3-set $S$ we have that $\DD|_S$ is not connected and hence $c_\DD(S)=0$. 
Hence $c_\DD$ agrees with $c_\GG$ on all sets of size $2$ and $3$ and by \Cref{lem: lindner} it follows that $c_\DD=c_\GG$. 
 If $c_\DD(S^\ast\cup \{i\})=1$ we must have a node $t\in S^\ast\cup \{i\}$ that is the child of everyone else. 
However $\DD|_{S^\ast}$ has no edges as neither $\GG|_{S^\ast}$ nor $\HH|_{S^\ast}$ has any edges, hence, as $|S^\ast|\geq 2$ we must have $t=i$. 
 Hence we must have all edges $j\to i$, for $j\in S^\ast$, in $\DD$. 
Left to check is that $c_\DD$ agrees with $c_\HH$ for all 3-sets on the form discussed above, that is $c_\DD(\{i,j,k\})=c_\HH(\{i,j,k\})=1$ for all $j\in S^\ast$ and $k\in\pa_\GG(i)\cup S^\ast$.  
\end{proof}
 
From this we get a linear upper bound of $\CIM_n$. 

%---PROPOSTION: Diameter of CIM polytope----
\begin{proposition}
\label{prop: diameter of cim polytope}
The diameter of $\CIM_n$ is less than or equal to $2n-2$. 
\end{proposition}

\begin{proof}
Let $\GG$ be any given DAG on $n$ nodes and let $v_1, \dots, v_n$ be a topological order of the vertices of $\GG$. 
Equivalently, if $v_i$ is a parent of $v_j$ then $i< j$. 
Let $\GG_n$ be the DAG with no edges. 
Define $\GG_{k-1}$ recursively to be $\GG_k$ with all edges $v\to k$ for $v\in\pa_\GG(k)$. 
Then, as $\pa_\GG(v_1)=\emptyset$, we must have that $\GG_2=\GG$. 
All that is left to show is that $\conv(\GG_{k-1}, \GG_k)$ is an edge of $\CIM_n$ for all $k$. 
This follows by \Cref{prop: realizing a set} and our observation regarding parents of $v_i$. 
Hence the distance from any vertex $c_\GG$ of $\CIM_n$ to the specific vertex $c_{\GG_n}$ is at most $n-1$, and the total diameter is at most twice that.
\end{proof}

It can be checked via \texttt{polymake} \cite{AGHJLPR17, GJ00} that $\diam\CIM_n=n-1$ for $n\in\{1,2,3,4\}$ and all these distances are realised by the graph distance between the DAG with no edges and the complete graph. 
This seems reasonable in light of the following proposition. 

\begin{proposition}
\label{prop: not edge of cim}
Let $\GG$ be a DAG, let $i_1,\dots,i_k$ be vertices and $S_1, \dots S_k$ be sets such that $\cl_\GG(i_t)\cap S_t=\emptyset$ for all $1\leq t\leq k$.
Let $\HH$ be obtained from $\GG$ via adding in all edges $s\to i_t$ for $s\in S_t$. 
If $\HH$ is a DAG and $k\geq 2$, then $\conv(c_\GG, c_\HH)$ is not an edge of $\CIM_n$.
\end{proposition}

To show this we will make use of the following lemma which is a fundamental fact from polytope theory.

\begin{lemma}
\label{lem: square not edge}
Let $P$ be a polytope and let $v$ be a vertex of $P$. 
If there exists non-zero vectors $u_1$ and $u_2$ such that $v+u_1$, $v+u_2$, and $v+u_1+u_2$ are all vertices of $P$, then $\conv(v, v+u_1+u_2)$ is not an edge of $P$. 
\end{lemma}
\begin{proof}[Proof of \cref{prop: not edge of cim}]
Let $\GG_1$ be that DAG that is obtained from $\GG$ via adding in all edges $s\to i_1$ where $s\in S_1$. 
Let $\GG_2$ be that DAG that is obtained from $\GG$ via adding in all edges $s\to i_t$ where $s\in S_t$ for $1<t\leq k$. 
As all of $\GG$, $\HH$, $\GG_1$, and $\GG_2$ have different skeleton, none are Markov equivalent. 
Thus the result follows from \cref{lem: square not edge} we we can show that $c_\GG+c_\HH=c_{\GG_1}+c_{\GG_2}$. 
By \cref{lem: lindner} it is enough to show $c_\GG(S)+c_\HH(S)=c_{\GG_1}(S)+c_{\GG_2}(S)$ for all sets $S$ such that $|S|\in \{2,3\}$. 
This equality follows directly for all 2-sets from \cref{lem: lindner}, as they encode the skeletons of the graphs. 
If $S$ is not a 3-set such that $i_t\in S$ and $S\cap S_t\neq\emptyset$ for some $t$, then $\GG|_S=\HH|_S=\GG_1|_S=\GG_2|_S$ and hence the equality holds. 
Thus we can assume that $S=\{i_t, s_t, p_t\}$ where $s_t\in S_t$ and $p_t\in S_t\cup\pa_\GG(i_t)$. 
The rest follows by definition of the characteristic imset and the construction of $\GG_1$ and $\GG_2$.  
\end{proof}
Hence adding in parents to several vertices at the same time is, in some sense, hard.
Therefore it would be reasonable that the distance between the empty and the complete graph will be $n-1$ for every $n$. 
The construction in the proof of \cref{prop: diameter of cim polytope} is however not always optimal, even if one graph is the empty graph.  

\begin{example}
Consider the graphs $\GG=([5],\emptyset)$ and $\HH=([5], E)$ where $\GG$ is the empty graph and $\HH$ is a star where the middle vertex has exactly $2$ parents. % as given in \cref{fig: cim diam ex}. 
Then the construction in the proof of \cref{prop: diameter of cim polytope} gives a path in $G(\CIM_n)$ of length $4$. 
However, utilising \cref{prop: realizing a set} we can move to a graph with the correct skeleton and utilizing  \cref{thm: edges of trees characterization}, we can move directly to $\HH$, see \cref{fig: cim diam ex}. 
However, \cref{prop: not edge of cim} shows that $\conv(c_\GG, c_\HH)$ is not an edge of $\CIM_6$ and hence the distance between $\GG$ and $\HH$ is 2.
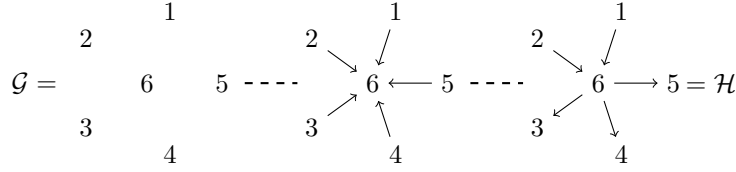
\begin{figure}
\[
\begin{tikzpicture}
\begin{scope}
\node at (-1.5, 0) {$\GG=$};
\node (p6) at (0:0) {$6$};

\foreach \x in {1,2,3,4,5}{
    \node (p\x) at ({\x*72}:1) {${\x}$};
}
\end{scope}
\begin{scope}[shift={(3,0)}]
\node (p6) at (0:0) {$6$};

\foreach \x in {1,2,3,4,5}{
    \node (p\x) at ({\x*72}:1) {$\x$};
    \draw[->] (p\x) -- (p6);
}
\end{scope}
\begin{scope}[shift={(6,0)}]
\node at (1.5, 0) {$=\HH$};
\node (p6) at (0:0) {$6$};
\foreach \x in {1,2,3,4,5}{
    \node (p\x) at ({\x*72}:1) {$\x$};
}
\foreach \from/\to in{1/6, 2/6, 6/3, 6/4, 6/5}{
    \draw[->] (p\from) -- (p\to);
}
\end{scope}
\draw[thick, dashed] (1.3, 0) -- (2, 0)
    (4.3, 0) -- (5, 0);
\end{tikzpicture}
\]
\caption{An example of a short path over $\CIM_n$ that is not used in the proof of \cref{prop: diameter of cim polytope}. Edges of $\CIM_n$ are denoted with dashed lines. }
\label{fig: cim diam ex}
\end{figure}
\end{example}

So far we have shown that we have quadratic bounds on the diameter of the faces $\CIM_G$, and if $G$ is a tree this bound becomes linear, in the number of vertices of $G$. 
We also have a linear bound on the diameter of the whole polytope $\CIM_n$. 
This leads us to believe that we in fact have a linear bound on $\CIM_G$ for any $G$. 

For any DAG $\GG$ with skeleton $G$ and vertex $i$ we let $\GG_{\downarrow i}$ be the graph where we have $k\to i$ for all $k\in \ne_\GG(i)$ and is otherwise identical to $\GG$. 
\begin{lemma}
If $\GG$ is a DAG, then $\GG_{\downarrow i}$ is a DAG for any vertex $i$. 
\end{lemma}

\begin{proof}
Any new cycles would have to use one of the edges that were reversed and thus pass through $i$, but $i$ has no outgoing edges in $\GG_{\downarrow i}$. 
Hence $\GG_{\downarrow i}$ has no directed cycles. 
\end{proof}

To then show a linear bound on $\diam\CIM_G$ it is enough to show the following conjecture.
\begin{conjecture}
\label{conj: sink conj}
If $\GG$ and $\GG_{\downarrow i}$ are not Markov equivalent, then $\conv\left(c_\GG, c_{\GG_{\downarrow i}}\right)$ is an edge of $\CIM_G$. 
\end{conjecture}

Let $\GG$ be a DAG with skeleton $G$. 
Define an order on $[n]$ as $v_1,\dots,v_n$ where we have $v_i\to v_j\in\GG$ implies $i<j$.
That is, take a topological order of $\GG$. 
Then for any DAG $\HH$ with skeleton $G$ we have $\GG=(\dots((\HH_{\downarrow v_2})_{\downarrow v_3})\dots)_{\downarrow v_n}$. 
Thus if \cref{conj: sink conj} holds, then $\diam\CIM_G\leq n-1$ for all graphs $G$.

%%%%%%%%%%%%%%%%%%%%%%%%%%%%%%%%%%%%%%%%%%%%%
%---SECTION: Discussion and further questions
\section{Discussion}
\label{sec: discussion}
In this paper we have shown that we have, at worst, quadratic bounds for $\diam\CIM_G$ and linear bounds for $\diam\CIM_n$. 
As the dimension of $\CIM_n$ is $2^n-n-1$ we get that the diameter of $\CIM_n$ grows at most quadratically in the logarithm of the dimension, significantly slower than any general bounds (for example \cite{Naddef89}).
In this sense we observe that $G(\CIM_n)$ is highly connected. 

For an edge-walk on $\CIM_G$ maximising an objective $W$, it is  not always the case that the optimal path monotone in $W$.
As seen in \cite{C02, LRS22,LRS20,TBA06}, not having access to all edges can still give us consistency guarantees with additional assumptions on $W$. 
However, when dealing with a score function based on data, as often is the case with $\CIM_n$, these additional assumptions are not guaranteed to hold for any finite sample size and, especially for smaller sample sizes, more edges can still improve performance \cite{LRS22}. 
As we can expect there to be many edges in $\CIM_n$ and $\CIM_G$, the fundamental question becomes which edges are the most crucial for performance and which ones are easily checked. 
For example, while the edges of \cref{prop: realizing a set} give us a way to transverse the polytope in few steps, in practice it may be easier to work with only the edges of \cref{thm: edge pair} as a repeated use can indeed reach the same graph. 
However, the question which specific class of edges to use cannot be properly discussed without a better understanding of the edge structure of $\CIM_n$ in general. 

\subsection*{Acknowledgements}
The author was partially supported by the Wallenberg AI, Autonomous Systems and Software Program (WASP) funded by the Knut and Alice Wallenberg Foundation.

%%%%%%%%%%%%%%%%%%%%%%%%%%
%---BIBLIOGRAPHY
\bibliographystyle{plain}
\bibliography{references}

\end{document}